\documentclass{ifacconf}

\usepackage{graphicx}      
\usepackage{natbib}        
\usepackage{amsmath}
\usepackage{amsfonts}
\usepackage{amssymb}
\usepackage{mathtools}
\usepackage{bm}
\usepackage{float}
\usepackage{import}
\usepackage{color}
\usepackage{url}

\usepackage{etoolbox}

\DeclareMathSizes{10}{9}{7}{5}
\AtBeginEnvironment{equation}{\small}
\AtBeginEnvironment{align}{\small}
\AtBeginEnvironment{gather}{\small}
\AtBeginEnvironment{multline}{\small}
\AtBeginEnvironment{flalign}{\small}

\newtheorem{remark}{Remark}
\newtheorem{assumption}{Assumption}
\newtheorem{theorem}{Theorem}
\newtheorem{proof}{Proof}
\newtheorem{lemma}{Lemma}

\newcommand{\IdentityMatrix}[1]{\bm{I}}
\newcommand{\ZerosMatrix}[2]{\ensuremath{\bm{0}}}

\newcommand{\GR}[1]{\textcolor{black}{#1}}

\newcommand{\DNN}[1]{\textcolor{black}{#1}}

\begin{document}
\begin{frontmatter}

\title{Optimal $\mathcal{W}_\infty$ Control of Prandtl-Ishlinskii Hysteresis Model via Weak Derivatives\thanksref{footnoteinfo}} 

\thanks[footnoteinfo]{This work was supported in part by the  CAPES through the Academic Excellence Program (PROEX), CNPq under Grants 317058/2023-1 and 422143/2023-5, FAPEMIG under Grant BPD-00960-22, and in part by Petrobras/ANP under Grants 2023/00494-5 and 2023/00643-0.}

\author[First]{Daniel N. Cardoso} 
\author[Second]{Petrus E. O. G. B. Abreu} 
\author[First,Third]{Guilherme V. Raffo}

\address[First]{Graduate Program in Electrical Engineering, Federal University of Minas Gerais, Belo Horizonte, Brazil\\(danielneri@ufmg.br, raffo@ufmg.br).}
\address[Second]{Graduate Program in Health Sciences: Infectious Diseases and Tropical Medicine, Federal University of Minas Gerais, Belo Horizonte, Brazil (petrusabreu@ufmg.br).}
\address[Third]{Department of Electronic Engineering, Federal University of Minas Gerais, Belo Horizonte, Brazil.}

\begin{abstract}                
\GR{This work proposes a novel robust optimal} $\mathcal{W}_\infty$ controller for \GR{dynamic systems with Prandtl-Ishlinskii hysteresis.} By \GR{utilizing} weighted Sobolev spaces $\mathcal{W}_{m,p,\Gamma}$, the approach uses \GR{weak derivatives to rigorously handle the} non-differentiable, input non-affine nature of hysteresis. This formulation recasts the Prandtl-Ishlinskii operator as a bounded uncertainty multiplying the input rate, enabling \GR{robust optimal} controller design via linear matrix inequalities, while guaranteeing $\mathcal{W}_{3,2,\Gamma}$-stability with a $\mathcal{W}_\infty$-gain bound. \GR{A numerical study on} a piezoelectric actuator model validates its effectiveness, demonstrating asymptotic tracking and the attenuation of both hysteresis and external disturbances through a straightforward implementation.

\end{abstract}


\end{frontmatter}
\sloppy
\section{Introduction}
\label{Sec:Introduction}
\vspace{-2mm}

\DNN{Hysteresis is a nonlinear phenomenon inherent to many dynamical systems and diverse engineering applications. Its key features include a memory effect, where output depends on the current and past input evolution; multiple steady-state equilibria under identical input conditions; and a highly nonlinear dependence on the input's magnitude and rate \GR{\citep{Morris2011}}.}
\DNN{These characteristics introduce significant challenges for both system identification and control \citep{Bequette1991, Pop_etl2018}. Common physical systems exhibiting hysteresis include magnetorheological dampers, piezoelectric actuators, and pneumatic control valves \citep{Abreu_etal2024}.}

\DNN{In control applications, unmodeled hysteresis degrades performance, causing oscillations, tracking errors, or even instability \citep{Rakotondrabe2013}. While phenomenological models like Bouc-Wen \citep{Wen1976} and Prandtl-Ishlinskii \citep{Kuhnen2003} effectively capture this behavior \citep{Ismail_etal2009}, they yield input non-affine and non-differentiable systems \citep{Abreu_etal2024}, complicating controller design and stability analysis \citep{Ikhouane_Rodellar2007}. Since these systems face external disturbances and modeling uncertainties, robust control is essential \citep{Hassani_etal2014}. Feedforward inverse compensation attempts to linearize the dynamics \citep{AlJanaideh_etal2011}; however, as \citet{Tao_Kokotovic1995} discuss, these methods demand precise, invertible models and are highly sensitive to parameter uncertainties, where minor errors can cause severe performance degradation or instability.}

\DNN{Adaptive control compensates for hysteresis via online parameter adjustment \citep{Tao_Kokotovic1995,Chen_etal2008}, but typically assumes a specific model structure and struggles with unmodeled dynamics. Robust control schemes address hysteresis-induced uncertainties \citep{Liu_etal2014_HD,Riccardi_etal2014}, yet the complex memory dependence complicates controller design, often requiring restrictive assumptions and yielding conservative results \citep{Gu_etal2016Survey}. These limitations necessitate a general mathematical framework capable of handling hysteresis while guaranteeing stability and performance. Thus, this work proposes a novel robust control strategy within the framework of weighted Sobolev spaces \citep{treves2016topological}. Denoted by $\mathcal{W}_{m,p}$, these spaces comprise functions in $\mathcal{L}_p$ whose weak derivatives up to order $m$ also belong to $\mathcal{L}_p$. This formulation appropriately handles non-differentiable and non-affine hysteresis models by employing weak derivatives in the design process, resulting in the proposed robust $\mathcal{W}_\infty$ controller.}

The use of Sobolev spaces in optimal control design is not a new concept. It has been introduced by \citet{aliyu2011extending} to incorporate the rate of the cost variable into the cost functional, aiming to improve transient performance. Specifically, the $\mathcal{W}_{1,2}$-norm of the cost variable \GR{has been} employed instead of the classical $\mathcal{L}_2$-norm to formulate the cost functional. By applying dynamic programming, the resulting optimal control problem \GR{has been} transformed into the associated Hamilton--Jacobi (HJ) equation. However, obtaining analytical solutions to this HJ partial differential equation (PDE) \GR{remains challenging} for general nonlinear systems.

In particular, \cite{aliyu2011extending} \GR{have claimed} that the HJ PDE arising from the $\mathcal{W}_\infty$ control formulation in the Sobolev space is analytically intractable due to its inherent complexity. To circumvent this issue, a backstepping-based approach \GR{has been} proposed to simplify the problem. However, the resulting controller proved to be structurally similar to the classical nonlinear $\mathcal{H}_\infty$ formulation, differing primarily by the addition of an integrator in the cost variable. Consequently, no significant improvement in transient performance \GR{has been} achieved.

\DNN{Subsequent research has developed practical solutions for the $\mathcal{W}_\infty$ framework. To solve the associated HJ equation, \cite{cardoso2018approximated} proposed a Galerkin-based approximation algorithm. \cite{AUT2019} demonstrated that this formulation enables component-wise tuning of the cost variable and its derivatives. Adopting a specific cost functional simplifies the HJ equation for mechanical systems, yielding analytical solutions. Compared to classical nonlinear $\mathcal{H}_\infty$ controllers, the $\mathcal{W}_\infty$ controller achieves superior transient performance, faster disturbance attenuation, inherent robustness, and straightforward implementation \citep{AUT2019}. Finally, \cite{danielECC2024} showed the framework offers a direct method to increase system order, effectively handling input non-affine dynamics.}

\GR{Therefore, the} main contribution of this work lies in the development of a robust $\mathcal{W}_\infty$ controller for hysteretic systems. We \GR{demonstrate} that the weak derivatives inherent to the Sobolev norm provide straightforward and effective means of handling the \GR{non-differentiable} and \GR{non-affine} Prandtl-Ishlinskii operator used \GR{for} hysteresis phenomena \GR{representation}. \GR{Specifically}, this operator is recast as a bounded uncertainty multiplying the input rate, \GR{which enables a synthesis based on linear matrix inequality (LMI)}. The resulting controller \GR{guarantees} $\mathcal{W}_{3,2,\Gamma}$-stability \GR{along} with an $\mathcal{W}_\infty$-gain bound.


\textbf{Notation and definitions:}  Italic lowercase letters denote scalars, boldface italic lowercase letters denote vectors, and boldface italic uppercase letters denote matrices. {\small$(\cdot)^T$\normalsize} and {\small$(\cdot)^{-1}$\normalsize} stand, respectively, for the transpose and inverse elements of \small$(\cdot)$\normalsize. \small$\mathbb{N} \triangleq \{1,2,...\}$\normalsize, \small$\mathbb{R} \triangleq (-\infty, \infty)$, $\mathbb{R}_{\geq 0} \triangleq [0, \infty )$\normalsize, \small$\mathbb{R}_{> 0} \triangleq (0, \infty )$\normalsize, \small$\mathbb{R}^n \triangleq \{\bm{r} = [r_1 \; ... \; r_n]':r_i \in \mathbb{R}\}$\normalsize, and \small$\mathbb{R}^{n\times m} \triangleq \{\bm{R} = [\bm{r}_1 \;... \;\bm{r}_m]:\bm{r}_i \in \mathbb{R}^n, i \in \{1, 2, \cdots, m\}\}$\normalsize. \small$\ZerosMatrix{n}{m}$ \normalsize and \small$\IdentityMatrix{n}$ \normalsize are, respectively, zero and identity matrices with appropriate dimensions. Let \small$t \in \mathbb{R}_{\geq 0}$ \normalsize denote the time variable and 
\small$\bm{z}(t): \mathbb{R}_{\geq 0} \rightarrow \mathbb{R}^{n_z}$ \normalsize be a time-varying function, then \small$\dot{\bm{z}}(t) \triangleq {d\bm{z}(t)}/{dt}$ \normalsize denotes its time derivative.
Let \small$p \in \mathbb{N}\cup\{\infty\}$, $m \in \mathbb{N}$, and $\bm{z}: \Psi \to \mathbb{R}^{n_z}$\normalsize. If \small$\bm{z}$ \normalsize belongs to the weighted Lebesgue space, i.e. \small$\bm{z} \in \mathcal{L}_{p,\bm{\Lambda}}[\Psi]$\normalsize, then its weighted \small$\mathcal{L}_{p}$\normalsize-norm is finite, \small$||\bm{z}||_{\mathcal{L}_{p,\bm{\Lambda}}} \triangleq \left(\int_{\Psi} ||\bm{\Lambda}^{1/p}\bm{z}||_p^p~d\Psi\right)^{{1}/{p}} < \infty$\normalsize, where \small$\bm{\Lambda}$ \normalsize is a positive definite symmetric matrix with appropriate dimension. If \small$\bm{z}$ \normalsize belongs to the weighted Sobolev space, i.e. \small$\bm{z} \in \mathcal{W}_{m,p,\bm{\Gamma}}[\Psi]$\normalsize, then \small$||\bm{z}||_{\mathcal{W}_{m,p,\bm{\Gamma}}} \triangleq \big(\sum_{\alpha = 0}^{m} ||{\partial^{\alpha}\bm{z}}||_{\mathcal{L}_{p,\bm{\Gamma}}}^p\big)^{{1}/{p}} < \infty$\normalsize,  with \small$\bm{\Gamma}  \triangleq \{\bm{\Gamma}_0, ..., \bm{\Gamma}_m\}$\normalsize, where \small${\partial^{\alpha}\bm{z}}$ \normalsize is the \small$\alpha$\normalsize-th weak derivative of \small$\bm{z}$\normalsize. It is well known that the weak derivative coincides with the classical derivative whenever the latter exists. The Lebesgue, \small$\mathcal{L}_{p}$, \normalsize and Sobolev, \small$\mathcal{W}_{m,p}$, \normalsize spaces are recovered as special cases when all weighting matrices are identity.

\vspace{-2mm}
\section{Revisiting $\mathcal{H}_\infty$ Control Toward a Weighted Sobolev Space Formulation}
\vspace{-2mm}
\label{Preliminaries}

Consider the non-autonomous system given by
\small\begin{align}
\label{GenericNonlinearSystem}
\mathcal{P}_1:\begin{cases}
    \dot{\bm{x}}(t) &= f(\bm{x},\bm{u},\bm{w},{t}), \\
    \bm{z}(t) &= h(\bm{x},\bm{u}),~~\bm{x}(0) = \bm{x}_0,
\end{cases}
\end{align}\normalsize
where $t \in \mathbb{R}_{\geq 0}$, $\bm{x}(t): \mathbb{R}_{\geq 0} \to \mathbb{R}^{n_x}$ is the state vector, $\bm{u}(t): \mathbb{R}_{\geq 0} \to \mathbb{R}^{n_u}$ is the control input, $\bm{w}(t): \mathbb{R}_{\geq 0} \to \mathbb{R}^{n_w}$ is the disturbance, and $\bm{z}(t): \mathbb{R}_{\geq 0} \to \mathbb{R}^{n_z}$ is the cost variable, with $n_x, n_u, n_w, n_z \in \mathbb{N}$. \GR{The function} $f(\cdot)$ denotes the vector field of the dynamical system\GR{, while} $h(\cdot)$ \GR{describes} the cost variable mapping.

Classic nonlinear $\mathcal{H}_\infty$ control seeks an admissible law $\bm{u}^* \in \mathcal{U}$ that solves
\vspace{-1mm}\small\begin{align}
    \min_{\bm{u} \in \mathcal{U}}\max_{\bm{w} \in \mathcal{W}} \left(||\bm{z}(t)||^2_{\mathcal{L}_{2}} - \gamma_{\mathcal{H}}^2 ||\bm{w}(t)||^2_{\mathcal{L}_2}\right),
\end{align}\normalsize
where $\mathcal{W} =\mathcal{L}_2[\mathbb{R}_{\geq 0})$. If it exists, $\bm{u}^*$ ensures $\mathcal{L}_2$-stability with gain $\gamma_{\mathcal{H}} > 0$ \citep{van1992sub}, guaranteeing \small$||\bm{z}(t)||^2_{\mathcal{L}_{2}} \leq \gamma_{\mathcal{H}}^2 ||\bm{w}(t)||^2_{\mathcal{L}_2}+ c_1$\normalsize, for all $\bm{w} \in \mathcal{W}$ and $c_1 > 0$. Minimizing $\gamma_{\mathcal{H}}$ restricts the maximum disturbance energy transferred to the cost variable, while $c_1$ captures the energy contribution from initial conditions. However, because $\mathcal{H}_\infty$ control ignores the cost variable's rate of change, the optimal $\mathcal{W}_\infty$ technique was introduced to address this limitation.

In the $\mathcal{W}_\infty$ framework, the cost functional utilizes a weighted Sobolev norm, deriving the optimal control $\bm{u}^*$ by solving
\small\begin{align}
\label{OptimalWinfControlProblem}
    \min_{\bm{u} \in \mathcal{U}} \max_{\bm{w} \in \mathcal{W}} \left(||\bm{z}(t)||^2_{\mathcal{W}_{m,p,\bm{\Gamma}}} {-} \gamma_{\mathcal{W}}^2 ||\bm{w}(t)||^2_{\mathcal{L}2}\right).
\end{align}\normalsize
If it exists, $\bm{u}^*$ ensures $\mathcal{W}_{m,p,\Gamma}$-stability with gain $\gamma_{\mathcal{W}} \geq 0$ \citep{AUT2019}, satisfying \small$||\bm{z}(t)||^2_{\mathcal{W}_{m,p,\bm{\Gamma}}} \leq \gamma_{\mathcal{W}}^2 ||\bm{w}(t)||^2_{\mathcal{L}_2} + c_2$\normalsize, for all $\bm{w} \in \mathcal{W}$ and $c_2 > 0$. Here, $\gamma_{\mathcal{W}}$ bounds the disturbance energy transferred to the cost variable and its weak derivatives. Because disturbances manifest in derivatives first, minimizing $\gamma_{\mathcal{W}}$ yields a highly responsive controller that accelerates attenuation and smooths transients, improving convergence, precision, and energy efficiency while reducing overshoot.


\vspace{-2mm}
\section{Mathematical Modeling of Systems with Hysteresis}
\vspace{-2mm}
\label{SubSec:Modeling}

A common approach to represent systems with hysteresis consists of a cascade structure composed of a dynamic model $\mathcal{G}$, describing the linear behavior of the system $\mathcal{S}$, preceded by a model $H(\cdot)$ that captures the hysteresis behavior.
A block diagram illustrating this structure is shown in Fig.~\ref{Fig:SubSecModel:PositioningSystem}.

\begin{figure}[H]
	\centering{
			\def\svgwidth{1\columnwidth}		
			\footnotesize{
					\import{Figures/}{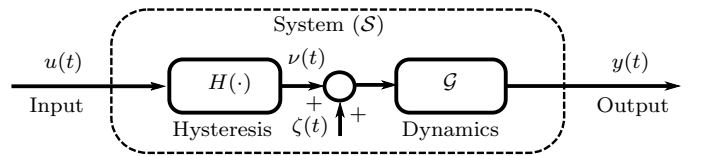} \vspace{-6mm}}
			\caption{Schematic representation of the system $\mathcal{S}$ with hysteresis nonlinearity $H(\cdot)$ and linear dynamics $\mathcal{G}$.}
			\label{Fig:SubSecModel:PositioningSystem}
		}
\end{figure}

\GR{Following established} modeling approaches, we assume the linear dynamics of the system \GR{are governed} by the state-space representation \citep{Riccardi_etal2013}
\small\begin{align}
\label{Eq:SubSecModel:HystereticDynamicSystem}
	\mathcal{G}:\begin{cases}
		\dot{d}(t)  &=  v(t), \\ 
		\dot{v}(t)  &=  b_0\big[\nu(t) {+} \zeta(t)\big]{-}a_0d(t) {-}a_1v(t),\\
		y(t)    &=  d(t),
	\end{cases}
\end{align}\normalsize
where $d: \mathbb{R}_{\geq 0} \to \mathbb{R}$ and $v: \mathbb{R}_{\geq 0} \to \mathbb{R}$ represent scalar functions.
The coefficients $a_0, \;a_1, \; b_0 \in \mathbb{R}$, with $b_0 \neq 0$, are constants. 
The variable $y: \mathbb{R}_{\geq 0} \to \mathbb{R}$ denotes the system output, and 
$\zeta: \mathbb{R}_{\geq 0} \to \mathbb{R}$ encapsulates the cumulative impact of external disturbances, 
parametric uncertainties, and unmodeled dynamics.
Meanwhile, $\nu: \mathbb{R}_{\geq 0} \to \mathbb{R}$ denotes the input of the linear dynamic model ${\cal G}$ 
and is obtained through the transformation of the system input $u: \mathbb{R}_{\geq 0} \to \mathbb{R}$ by the function $H(\cdot)$, which characterizes the hysteresis behavior.

\begin{assumption}
    The state vector $\bm{x} = [d \;\; v]^T$ is available for control design purposes.
\end{assumption}

In this work, the hysteresis behavior $H(\cdot)$ is modeled using the Prandtl-Ishlinskii operator, which is defined as a weighted sum of elementary components, often referred to as play operators \citep{Kuhnen2003}. \GR{Fig. \ref{Fig:SubSecModel:PlayOperator} provides a graphical description of the input-output relationship  for this operator}. Each play operator \GR{$P_{r_i}:\mathbb{R}_{\geq 0} \to \mathbb{R}$} is characterized by a radius $r_i \in \mathbb{R}_{\geq 0}$\GR{. For a continuous input $u(t)$, the $i$-th play operator is mathematically defined as}
%
\small\begin{align}
	\nu_{{r}_i}(t) =& \max\!\big(\!\min(u(t)+r_i,\,\nu_{{r}_i}(t^{-})),\,u(t)-r_i\big), \nonumber \\[1mm]
	\triangleq& P_{{r}_i}\big(u(t); \nu_{{r}_i}(t^{-})\big), \label{Eq:SubSecModel:PlayOperator}
\end{align}\normalsize
where $\nu_{{r}_i}(t)$ and $\nu_{{r}_i}(t^{-})$ represent the state of the play operator at time $t$ and immediately before $t^{-}$, respectively. 

\begin{figure}[H]
	\centering{
		\def\svgwidth{0.5\columnwidth}		
		\scriptsize{
			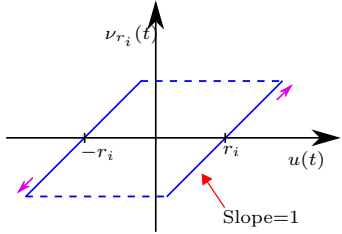\vspace{-2mm}}  
		\caption{Input-output mapping of the play operator $P_{{r}_i}$\GR{where} $u(t)$ is the input signal, and $\nu_{r_{i}}(t)$ denotes the state values of the $i$-th operator with radius $r_{i}$.}
		\label{Fig:SubSecModel:PlayOperator}
	}
\end{figure}

\GR{T}he mapping illustrated in Fig.\,\ref{Fig:SubSecModel:PlayOperator}, which corresponds to \GR{the} play operator \eqref{Eq:SubSecModel:PlayOperator}, \GR{identifies} two distinct modes that describe \GR{the evolution operator's} state over time. \GR{The first mode is the linear region}, where the state varies linearly, with a \GR{possible} translation, \GR{given} by $\nu_{{r}_i}(t){=}u(t) \pm r_i$, and represented by the solid blue lines in Fig.\,\ref{Fig:SubSecModel:PlayOperator}. The \GR{second} mode, known as the play region, is characterized by a constant state value, $\nu_{{r}_i}(t){=}\nu_{{r}_i}(t^{-})$, and \GR{is} indicated by dashed lines in Fig. \ref{Fig:SubSecModel:PlayOperator} \citep{Esbrook_etal2014}. These two modes are critical \GR{for} structuring the optimal control formulation presented in Section \ref{Subsec:PositioningSystem:ModelingDetails}.

\GR{Considering the Prandtl-Ishlinskii operator composed of $n_{\rm r}$ play operators $P_{r_{i}}$, as defined in (\ref{Eq:SubSecModel:PlayOperator}), }the output of $H(\cdot)$ can be expressed as
\small\begin{equation}
	\label{Eq:SubSecModel:PIOperator}
	\nu(t) = H\big(u(t);\, \bm{\nu}_{{r}}(t^{-})\big)=\textstyle\sum_{i=1}^{n_{\rm r}}\theta_iP_{r_{i}}\big(u(t);\, \nu_{r_{i}}(t^{-})\big),
\end{equation}\normalsize
where the vector $\bm{\nu}_{{r}}(t) \triangleq [\nu_{r_{1}}(t) \; \nu_{r_{2}}(t) \;\cdots \; \nu_{r_{n_{\rm r}}}(t)]^T$
comprises the state variables $\nu_{r_{i}}(t)$ of the $n_{\rm r}$ play operators at the current time $t$, and the vector
$\bm{\nu}_{{r}}(t^{-}) \triangleq [\nu_{r_{1}}(t^{-})\; \nu_{r_{2}}(t^{-})\; \cdots\; \nu_{r_{n_{\rm r}}}(t^{-})]^T$
represents the state variables at the instant immediately preceding $t$.
\begin{assumption}
Each weight $\theta_i \in \mathbb{R}_{>0}$ is bounded and positive.
\end{assumption}
\begin{assumption}
\label{Assumption3}
    The radii $r_{i} \in \mathbb{R}_{\geq 0}$ satisfy $0 = r_{1} < r_{2} < \cdots < r_{n_{\rm r}} < \infty$.
\end{assumption}
 
\begin{remark}
\label{RemarkCtrb}
    Assumption \ref{Assumption3} ensures that $r_{1}=0$, \GR{which implies} $\nu_{{r}_1}(t){=}u(t)$ for all $t \in \mathbb{R}_{\geq 0}$. Consequently, at any time instant, at least one operator in the model \eqref{Eq:SubSecModel:PIOperator} \GR{has its} state within the linear region. This condition is crucial \GR{for guaranteeing} the controllability of the system \citep{Esbrook_etal2014}.
\end{remark}


Then, we define
\small
\begin{equation}
	\label{Eq:SubsecModel:PlayOperatorVector}
	\bm{\nu}_{{r}}(t) \triangleq \mathcal{P}\big(u(t);\, \bm{\nu}_{{r}}(t^{-})\big) = [P_{r_{1}}\; P_{r_{2}}\; \cdots\; P_{r_{n_{\rm r}}}]^T,
\end{equation}
\normalsize
which allows \eqref{Eq:SubSecModel:PIOperator} to be written compactly as
\small
\begin{equation}
	\label{Eq:SubsecModel:PIOperatorVector}
	\nu(t) = H\big(u(t);\, \bm{\nu}_{{r}}(t^{-})\big) = \bm{\theta}^T\bm{\nu}_{{r}}(t),
\end{equation}
\normalsize
where $\bm{\theta} \triangleq [\theta_1\; \theta_2\; \cdots\; \theta_{n_{\rm r}}]^T$.

\vspace{-2mm}
\GR{\section{Robust Optimal $\mathcal{W}_\infty$ Control Design}}
\vspace{-2mm}
\label{Subsec:PositioningSystem:ModelingDetails}


%
This section formulates the \GR{robust optimal } $\mathcal{W}_\infty$ control problem to asymptotically stabilize system (\ref{Eq:SubSecModel:HystereticDynamicSystem}) at a desired reference, mitigate hysteresis nonlinearity, and attenuate external disturbances. By \GR{introducing} the integral state $\tilde{\epsilon}(t)\triangleq\int_{0}^{t}\tilde{d}(\vartheta)d\vartheta$ to counteract \GR{constant} disturbances, parametric uncertainties, and unmodeled dynamics, system \eqref{Eq:SubSecModel:HystereticDynamicSystem} can be expressed via tracking error dynamics as follows:
\small\begin{align}
\label{Eq:SecControl:ErrorDynamicModel}
	\mathcal{S}_1{:} \begin{cases}
		\dot{\tilde{\epsilon}}(t)  =  \tilde{d}(t),\\ 
		\dot{\tilde{d}}(t)  = v(t),\\
		\dot{v}(t)  =  b_0\big[\nu(t) {+} \zeta(t)\big] {-}a_0\big[ \tilde{d}(t){+}d_{\rm r}\big] {-}a_1v(t), \\
            y(t) = \tilde{d}(t) + d_{\rm r},
	\end{cases}
\end{align}\normalsize
where $\tilde{d}(t)\triangleq d(t)-d_{\rm r}$ denotes the deviation of the variable $d(t)$ from the desired reference $d_{\rm r} \in \mathbb{R}$.

Considering the system \eqref{Eq:SecControl:ErrorDynamicModel} and adopting the $\mathcal{W}_\infty$ framework, the optimal control problem is formulated to determine the \GR{following} control law:
\begin{align}
\label{OptimalControlProblem}
    u^* = \arg{\min_{u \in \mathcal{U}}}\max_{\zeta \in \mathcal{W}} \mathcal{J}_{\mathcal{W}_\infty},
\end{align}
with the cost functional defined as
\small
\begin{align}
\label{CostFunctional}
    \mathcal{J}_{\mathcal{W}_\infty} \triangleq ||z(t)||^2_{\mathcal{W}_{3,2,\bm{\Gamma}}} - \gamma^2 ||\zeta(t)||^2_{\mathcal{W}_{1,2}},
\end{align}
\normalsize
where $z(t)=\tilde{d}(t)$, $\mathcal{U}: \mathbb{R}_{\geq 0} \mapsto \mathbb{R}^{n_u}$, and $ \mathcal{W}\subseteq \mathcal{W}_{1,2}[\mathbb{R}_{\geq 0})$.
The set $\bm{\Gamma}{=}\{\Gamma_0,\Gamma_1,\Gamma_2,\Gamma_3\}$ represents weighting matrices that must be specified.

\begin{lemma}
\label{Lemma1}
    The optimal control law \eqref{OptimalControlProblem} guarantees that the closed-loop system
    (\ref{Eq:SecControl:ErrorDynamicModel}) satisfies 
    \small
   \begin{align}
        ||z(t)||^2_{\mathcal{W}_{3,2,\bm{\Gamma}}} &\leq \gamma^2 ||\zeta(t)||^2_{\mathcal{W}_{1,2}} + c, \label{CostFunctional22}
    \end{align}
    \normalsize
    for some $c \in \mathbb{R}$. Therefore, the closed-loop system is $\mathcal{W}_{3,2,\bm{\Gamma}}$-stable
    with $\mathcal{W}_{3,2,\bm{\Gamma}}$-gain $\gamma$.
\end{lemma}

\begin{proof}
    \GR{Assuming} the control problem is feasible\GR{, which guarantees the existence of} an optimal control law \eqref{OptimalControlProblem}, there exists an optimal cost $\mathcal{J}^*_{\mathcal{W}_\infty} \in \mathbb{R}$\GR{. This cost is} associated with an optimal control law $u^*:\mathbb{R}_{\geq 0} \to \mathbb{R}$ and a worst-case disturbance $\zeta^*:\mathbb{R}_{\geq 0} \to \mathbb{R}$, $\zeta^* \in \mathcal{W}_{1,2}[\mathbb{R}_{\geq 0})$, such that
    \small
\begin{align}
\nonumber
\mathcal{J}^*_{\mathcal{W}_\infty} &= {\min_{\partial u \in \mathcal{U}}}\max_{\partial \zeta \in \mathcal{W}} \left(||z(t)||^2_{\mathcal{W}_{3,2,\bm{\Gamma}}} - \gamma^2 ||\zeta(t)||^2_{\mathcal{W}_{1,2}}\right), \\
\label{CostFunctional1}
    &= ||z^*(t)||^2_{\mathcal{W}_{3,2,\bm{\Gamma}}} - \gamma^2 ||\zeta^*(t)||^2_{\mathcal{W}_{1,2}},
\end{align}
\normalsize
where ${z}^*$ denotes the optimized cost variable. Since $\zeta^*$ is the worst-case disturbance, i.e., the one that maximizes the cost functional, for any generic disturbance $\zeta \in \mathcal{W}_{1,2}[\mathbb{R}_{\geq 0})$, the following inequality holds:
\small
\begin{align}
\nonumber
    \mathcal{J}^*_{\mathcal{W}_\infty} &\geq ||z^*(t)||^2_{\mathcal{W}_{3,2,\bm{\Gamma}}} - \gamma^2 ||\zeta(t)||^2_{\mathcal{W}_{1,2}}, \\
    ||z^*(t)||^2_{\mathcal{W}_{3,2,\bm{\Gamma}}} &\leq \gamma^2 ||\zeta(t)||^2_{\mathcal{W}_{1,2}} + \mathcal{J}^*_{\mathcal{W}_\infty}. \label{CostFunctional2}
\end{align}
\normalsize
Therefore, according to Definition 2 in \citet{AUT2019}, the closed-loop system is $\mathcal{W}_{3,2,\bm{\Gamma}}$-stable with $\mathcal{W}_{3,2,\bm{\Gamma}}$-gain $\gamma$. \qed
\end{proof}

\begin{remark}
    The cost functional \eqref{CostFunctional} is defined in terms of the weighted Sobolev norm of $z(t)$. As will be shown next, this feature allows us to carry out the control design using the weak derivative of the Prandtl-Ishlinskii operator.
\end{remark}

\GR{Now}, taking into account \eqref{Eq:SecControl:ErrorDynamicModel}, we introduce the following variable transformation\footnote{For simplicity, throughout the manuscript some functional dependencies are omitted.}:
\small
\begin{align}
\label{StatesTransformedSystem}
    z &=\tilde{d}, \\
    \partial^1 z &= \partial^1\tilde{d} = v,\\
    \partial^2 z &= \partial^2\tilde{d} = \partial^1 v = \eta,\\
    \partial^3 z &= \partial^3\tilde{d} = \partial^2 v = \partial^1 \eta = -a_0 v {-} a_1 \eta {+} b_0 \partial^{1} \left(\nu {+} \zeta\right), 
    \label{StatesTransformedSystem2}
\end{align}
\normalsize
where, in view of \eqref{Eq:SubsecModel:PIOperatorVector}, we obtain
\small
\begin{align}
	\label{Eq:SecControl:DerivativePIOperatorVector} 
	\partial^{1}\Big(\nu + \zeta\Big) &=  \partial^{1}\Big( H\big( u(t);\, \bm{\nu}_{{r}}(t^{-}) \big)\Big) + \partial^{1}\zeta,\nonumber \\
	&= \bm{\theta}^T\partial^{1}{\bm{\nu}}_{r}(t)+ \partial^{1}\zeta,
\end{align}
\normalsize
with $\bm{\nu}_{r}(t)$ given in (\ref{Eq:SubsecModel:PlayOperatorVector})\GR{. Thus}, its weak derivative can be expressed as
\begin{gather}
\label{EquationDerivativeOperator}
	\partial^{1}{\bm{\nu}}_{r}(t) =
	\begin{bmatrix}
		\partial^{1}{P}_{r_1}\big(u(t);\, \nu_{r_1}(t^{-})\big) \\
		\vdots  \\
		\partial^{1}{P}_{r_{n_{\rm r}}}\big(u(t);\, \nu_{r_{n_{\rm r}}}(t^{-})\big)
	\end{bmatrix} \!\!.
\end{gather}
The weak derivative of the Prandtl-Ishlinskii operator ${P}_{r_i}$ is derived in the following lemma.

\begin{lemma}
\label{Lemma2}
    Let ${P}_{r_i}(t):\mathbb{R}_{\geq 0} \to \mathbb{R}$ denote the Prandtl-Ishlinskii operator defined in \eqref{Eq:SubSecModel:PlayOperator}, where $t \in \mathbb{R}_{\geq 0}$ \GR{is} the time instant. Partition the time domain $\mathbb{R}_{\geq 0}$ into two disjoint subsets, $\Pi_i$ and $\Pi_i^c$, such that $\mathbb{R}_{\geq 0} = \Pi_i \cup \Pi_i^c$\GR{. Let} $\Pi_i$ denote the set of time instants during which the play operator $P_{r_i}$ \GR{is} in its linear region, \GR{with} $\Pi_i^c$ \GR{being} its complement. Then, the weak derivative of ${P}_{r_i}$ is given by
    \begin{align}
    \label{Eq:SecControl:DerivativePlayOperator}
	\partial^{1}{P}_{r_i}(t) = \begin{cases}
		\dot{u}, ~\text{if}~P_{r_i}  \in \Pi_i,     \\ 
		0,      ~ \text{if}~P_{r_i}  \in  \Pi_i^c.
	\end{cases}
    \end{align}
\end{lemma}

\begin{proof}
By definition, the weak derivative of a generic function $f:\Omega \to \mathbb{R}$, $f \in \mathcal{L}_1[\Omega]$, $\Omega \subseteq \mathbb{R}_{\geq 0}$, is a function $g:\Omega \to \mathbb{R}$ that satisfies \citep{knabner2003numerical}
\begin{align}
\label{WeakDerivative}
    \int_{\Omega} f(t)\dfrac{d\bm{\psi}(t)}{dt}dt &= - \int_{\Omega} g(t) \bm{\psi}(t) dt,
\end{align}
for every $\bm{\psi}(t):\Omega \to \mathbb{R}$ with compact support in $\Omega$.

As described in Section \ref{SubSec:Modeling} and illustrated in Fig.~\ref{Fig:SubSecModel:PlayOperator}, the state of the play operator \eqref{Eq:SubSecModel:PlayOperator} alternates between the linear and play regions. The weak derivative of a play operator can then be obtained by considering two sets that encompass these regions \citep{Esbrook_etal2014}. Defining $\Pi_i$ as the set of \GR{time instants} where $P_{r_i}(t)$ \GR{is} in the linear region and $\Pi_i^c$ as \GR{its} complement, with $\mathbb{R}_{\geq 0} = \Pi_i \cup \Pi_i^c$, one can split the following integral into two parts:
\begin{align}
\label{EqSplitDomain}
     \int_{\mathbb{R}_{\geq 0}}P_{r_i}\dfrac{d\bm{\psi}}{dt}dt = \int_{\Pi_i}P_{r_i}\dfrac{d\bm{\psi}}{dt}dt {+} \int_{\Pi_i^c}P_{r_i}\dfrac{d\bm{\psi}}{dt}dt.
\end{align}
\GR{Given that} the weak derivative of a function \GR{coincides with} its ordinary derivative whenever the latter exists, one may propose \eqref{Eq:SecControl:DerivativePlayOperator} as the weak derivative of $P_{r_i}$, \GR{resulting in}
\begin{align}
\label{IntegrationByParts1}
    \int_{\Pi_i}P_{r_i}\dfrac{d\bm{\psi}}{dt}dt &= - \int_{\Pi_i}\partial^{1}P_{r_i}\,{\bm{\psi}}\,dt = - \int_{\Pi_i}\dot{u}\,{\bm{\psi}}\,dt,\\
    \label{IntegrationByParts2}
    \int_{\Pi_i^c}P_{r_i}\dfrac{d\bm{\psi}}{dt}dt &= - \int_{\Pi_i^c}\partial^{1}P_{r_i}\,{\bm{\psi}}\,dt = 0.
\end{align}
Consequently, from \eqref{EqSplitDomain}, \eqref{IntegrationByParts1}, and \eqref{IntegrationByParts2}, it follows that
\begin{align}
     \int_{\mathbb{R}_{\geq 0}}P_{r_i}\dfrac{d\bm{\psi}}{dt}dt = - \int_{\Pi_i}\dot{u}\,{\bm{\psi}}\,dt, 
\end{align}
which proves that \eqref{Eq:SecControl:DerivativePlayOperator} is the weak derivative of the Prandtl-Ishlinskii operator \eqref{Eq:SubSecModel:PlayOperator}. \qed
\end{proof}

\GR{Consequently, from} \eqref{Eq:SecControl:DerivativePIOperatorVector}, we have
\begin{align}
\label{WeakDerivativeV}
	\partial^1\nu(t) &=  \Big[\theta_1\; \theta_2 \cdots \theta_{n_{\rm r}}\Big]\Big[\partial^1{P}_{r_1} \;\partial^1{P}_{r_2} \cdots \partial^1{P}_{r_{n_{\rm r}}}\Big]^T,
\end{align}
which, based on \eqref{EquationDerivativeOperator} and \eqref{Eq:SecControl:DerivativePlayOperator}, can be rewritten as
\begin{align}
\label{DerivativeControlInputs}
    	\partial^1\nu(t) =  \bar{\theta}(t)\,\partial^1 u(t),
\end{align}
where $\bar{\theta}(t)$ is defined as the sum of the weights of all play operators that lie in the linear region at time $t$\GR{, that is,} 
%
	$\bar{\theta}(t) = \big[\sum_{i}\theta_i \;:\; P_{r_i} \in \Pi_i \big]\!$.
%
\GR{It is worth mentioning} that, for the hysteresis operator \eqref{Eq:SubsecModel:PIOperatorVector}, $\bar{\theta}(t)$ can only take values within a bounded interval, as detailed in \citep{Abreu_etal2024_robust}.
Consequently, in the context of robust control design, $\bar{\theta}(t)$ can be treated as a bounded uncertain parameter\GR{, in which}
\begin{align}
	\label{Eq:SecControl:IntervalWeightSet}
	\bar{\theta}(t) \in \Big[\bar{\theta}_{\min}, ~\bar{\theta}_{\max}\Big]\!.
\end{align}

 \begin{remark}
    When only one operator lies within the linear region, as discussed in Remark \ref{RemarkCtrb}, the bounded uncertain parameter attains its minimum value, $\bar{\theta}(t) = \bar{\theta}_{\min} = \theta_{1}$. Conversely, when all operators are within the linear region\GR{, it} reaches its maximum value, $\bar{\theta}(t) = \bar{\theta}_{\max} = \sum_{i=1}^{n_{\rm r}}\theta_i$.
 \end{remark}

\GR{By applying} the variable transformation \eqref{StatesTransformedSystem}–\eqref{StatesTransformedSystem2} to system \eqref{Eq:SecControl:ErrorDynamicModel}, and defining \GR{the state vector}
\[
\bm{\xi} {=} \begin{bmatrix}
		\tilde{d}  &
		v &
		\eta
	\end{bmatrix}^T,
\]
the tracking-error dynamics can be written \GR{in the following form, utilizing \eqref{DerivativeControlInputs}–\eqref{Eq:SecControl:IntervalWeightSet}:}
\begin{align}
\label{Eq:SecControl:ErrorModelClosedLoop} 
	\partial^1\bm{\xi}(t) = \bm{\tilde{A}}\bm{\xi}(t) + \bm{\tilde{B}}\big(\bar{\theta}\,\partial^1{u}(t) + \partial^{1}\zeta(t)\big),
\end{align}
\GR{where}
\vspace{-2mm}
\begin{align*}
	\bm{\tilde{A}} {=} \begin{bmatrix}
		0  & 1 & 0 \\
		0  & 0 & 1 \\
		0  & {-}a_0 & {-}a_1
	\end{bmatrix},\quad
	\bm{\tilde{B}} {=} \begin{bmatrix}
		0  \\
		0 \\
		b_0
	\end{bmatrix},
\end{align*}
\GR{and} the uncertain parameter $\bar{\theta}(t)$ belongs to the bounded interval
\eqref{Eq:SecControl:IntervalWeightSet}.

\begin{theorem}
\label{Teorema1}
    Suppose $(\bm{X},\bm{Y})$, with $\bm{X} \in \mathbb{R}^{3 \times 3}$ and $\bm{Y} \in \mathbb{R}^{1 \times 3}$, is a solution \GR{to} the following semidefinite program:
        \begin{gather}
        \label{LMIWinf}
        \min_{\bm{X},\bm{Y}} \gamma^*, \\
    \text{s.t. }\begin{cases}   
    \bm{X} > 0, \\
    \begin{bmatrix}
        \bm{\Psi}_1 \!\!&\bm{\tilde{B}}\!\!& \bm{X}\!\! & \!\!\ZerosMatrix{}{} \\
        * & {-}\gamma^* & \bm{\lambda}\bm{X} {+} \bar{\theta}b_0\bm{Y}\!\! & \!\!b_0 \\
        * & * & {-}\bm{\Psi}_2^{-1}\!\! & \!\!\ZerosMatrix{}{} \\
        * & * & * \!\!& \!\!{-}\bm{\Gamma}_3^{-1}
    \end{bmatrix} < 0,\; \forall \bar{\theta} \in \Big[\bar{\theta}_{\min}, \bar{\theta}_{\max}\Big],
    \end{cases} \nonumber
    \end{gather}
    where the entries $*$ are determined by symmetry, $\bm{\Psi}_1 \triangleq \bm{X}\bm{\tilde{A}}^T + \bm{\tilde{A}}\bm{X} + \bm{Y}^T\bm{\tilde{B}}^T \bar{\theta} + \bar{\theta}\bm{\tilde{B}} \bm{Y}$, $\bm{\Psi}_2 \triangleq \text{blkdiag}(\Gamma_0,\Gamma_1,\Gamma_2)$, and $\bm{\lambda} \triangleq \begin{bmatrix} 0 & -a_0 & -a_1 \end{bmatrix}$. Moreover, assume $\zeta(t) \in  \mathcal{W}_{1,2}[\mathbb{R}_{\geq 0})$. Then\GR{,} the optimal control law
    \begin{align}
        \label{IntOptCtrlLaw}
        \bm{u}^*(t) = \bm{Y}\bm{X}^{-1}\bm{x}(t) = \bm{K}\bm{x}(t),
    \end{align}
    applied to system \eqref{Eq:SecControl:ErrorModelClosedLoop}, with $\bm{x} = [\tilde{\epsilon}(t) \; \tilde{d}(t) \; v(t)]^T$, satisfies \eqref{CostFunctional22} with $\gamma^2 = \gamma^*$ and guarantees the asymptotic stability of the closed-loop system. 
\end{theorem}

\begin{proof}
    Consider
    \begin{align}
        \mathcal{V}(t) = \bm{\xi}^T(t)\bm{X}^{-1}\bm{\xi}(t), \; \bm{X}^{-1} > 0, \label{Lyap1}
    \end{align}
    such that
    \begin{align}
        \label{Eq1OptCtrl}
        \partial^1\mathcal{V} &= \partial^1{\bm{\xi}}^T\bm{X}^{-1}\bm{\xi} + \bm{\xi}^T\bm{X}^{-1}\partial^1{\bm{\xi}} \\[0.5mm]
        &< -z^T\Gamma_0z - \big(\partial^1{z}\big)^T\Gamma_1\partial^1{z}
        - \big(\partial^2{z}\big)^T\Gamma_2\partial^2{z}\nonumber \\[-0.5mm]
        &- \big(\partial^3{z}\big)^T\Gamma_3\partial^3{z} + \gamma^* \big(\partial^1\zeta\big)^T\partial^1\zeta, \nonumber 
    \end{align}
    where $z$, $\partial^1{z}$, $\partial^2{z}$, and $\partial^3{z}$ are given in \eqref{StatesTransformedSystem}–\eqref{StatesTransformedSystem2}.
    Using the results of Lemma \ref{Lemma2} and proposing \GR{the optimal input rate as the linear state feedback control law}
    \begin{align}
        \partial^1\bm{u}^* = \bm{K}\bm{\xi}, \label{LeiCtrlOpt}    
    \end{align}
    it follows from \eqref{StatesTransformedSystem}–\eqref{StatesTransformedSystem2}, \eqref{DerivativeControlInputs}, and \eqref{Eq:SecControl:ErrorModelClosedLoop} that
    \begin{align}
        \partial^1{\bm{\xi}}(t) &=  \left(\bm{\tilde{A}} + \bm{\tilde{B}}\bar{\theta}\bm{K}\right)\bm{\xi}(t) + \bm{\tilde{B}}\partial^{1}\zeta(t), \label{Eq2OptCtrl} \\
        \partial^3{z} &= \big(\bm{\lambda}  + b_0\bar{\theta}\bm{K}\big)\bm{\xi} + b_0\partial^{1}\zeta(t), \label{Eq3OptCtrl}
    \end{align}
    with $\bm{\lambda} \triangleq \begin{bmatrix} 0  & {-}a_0 & {-}a_1\end{bmatrix}$. Using \eqref{Eq2OptCtrl} and \eqref{Eq3OptCtrl}, we expand \eqref{Eq1OptCtrl} as
   \begin{align}
    \nonumber&\bm{\xi}^T\!\left(\!\left(\bm{\tilde{A}} + \bm{\tilde{B}}\bar{\theta}\bm{K}\right)^T\bm{X}^{-1} + \bm{X}^{-1}\left(\bm{\tilde{A}} + \bm{\tilde{B}}\bar{\theta}\bm{K}\right) \!\right)\!\bm{\xi} \\
    &+  \big(\bm{\tilde{B}}\partial^1\zeta\big)^T \bm{X}^{-1}\bm{\xi} + \bm{\xi}^T\bm{X}^{-1}\bm{\tilde{B}}\partial^1\zeta \nonumber\\
        &< -\bm{\xi}^T\!\left(\!\big(\bm{\lambda}  {+} b_0\bar{\theta}\bm{K}\big)^T\Gamma_3\big(\bm{\lambda}  {+} b_0\bar{\theta}\bm{K}\big) + \bm{\Psi}_2\!\right)\!\bm{\xi} \nonumber\\
        &{-}\bm{\xi}^T\!\left(\!\left( \bm{\lambda}  {+} b_0\bar{\theta}\bm{K}\right)^T\Gamma_3b_0\!\right)\partial^{1}\zeta {-} \left(\partial^{1}\zeta\right)^T\!\left(\!b_0\Gamma_3\left( \bm{\lambda}  {+} b_0\bar{\theta}\bm{K}\right)\!\right)\bm{\xi} \nonumber\\
        & {+} \big(\partial^1\zeta\big)^T\!\left(\gamma^* {-}  b_0\Gamma_3 b_0\right)\!\partial^1\zeta < 0\GR{.} \label{Eq4OptCtrl}
    \end{align}
    \GR{Then,} by grouping the terms in \eqref{Eq4OptCtrl}\GR{,} we obtain the inequality
     \begin{align}
    \begin{bmatrix}
        \bm{\Psi}_3 \;\; & \;\; \bm{X}^{-1}\bm{\tilde{B}} + \left( \bm{\lambda}  {+} b_0\bar{\theta}\bm{K}\right)^T\Gamma_3b_0\\
        * & -\gamma^* +  b_0\Gamma_3 b_0
    \end{bmatrix} < 0,
    \end{align}
with $\bm{\Psi}_3 \triangleq \left(\bm{\tilde{A}} + \bm{\tilde{B}}\bar{\theta}\bm{K}\right)^T\bm{X}^{-1} + \bm{X}^{-1}\left(\bm{\tilde{A}} + \bm{\tilde{B}}\bar{\theta}\bm{K}\right) + \bm{\Psi}_2 + \big(\bm{\lambda}  + b_0\bar{\theta}\bm{K}\big)^T\Gamma_3\big(\bm{\lambda}  + b_0\bar{\theta}\bm{K}\big)$.
    
    \GR{By considering} a congruence transformation \GR{using} $\text{blkdiag}(\bm{X},\;\IdentityMatrix{})$, \GR{applying} the change of variables $\bm{Y} = \bm{K}\bm{X}$, and \GR{utilizing} the Schur complement \GR{on} \eqref{Eq4OptCtrl}, we derive the semidefinite program \eqref{LMIWinf}. Its solution guarantees that \GR{both} \eqref{Lyap1} and \eqref{Eq1OptCtrl} are satisfied. The \GR{rate of the} optimal control law is expressed in terms of the weak derivative $\partial^1\bm{u}^*$ of $\bm{u}^*$ \GR{in \eqref{LeiCtrlOpt}}. Therefore, integrating both sides of \eqref{LeiCtrlOpt} over \GR{the interval} $[0,\;t]$ yields \eqref{IntOptCtrlLaw}.
    
    To show that the optimal control law $\bm{u}^*$ obtained from \eqref{LMIWinf} constitutes a \GR{robust} $\mathcal{W}_\infty$ controller and guarantees the asymptotic stability of the closed-loop system at a desired reference value, we integrate both sides of \eqref{Eq1OptCtrl} over $\mathbb{R}_{\geq 0}$ to obtain
    \begin{align}
         \lim_{t \to \infty}\mathcal{V}(t) - \mathcal{V}(0) <& -||z(t)||^2_{\mathcal{W}_{3,2,\bm{\Gamma}}} + \gamma^2 \big\| \partial^1\zeta(t) \big\|^2_{\mathcal{L}_2}, \\[0.5mm]
         ||z(t)||^2_{\mathcal{W}_{3,2,\bm{\Gamma}}} <& \gamma^2 \big\| \partial^1\zeta(t) \big\|^2_{\mathcal{L}_2} + \mathcal{J}^*_{\mathcal{W}_\infty},
        \label{IneqStability}
    \end{align}
    where, according to Lemma \ref{Lemma1}, $\mathcal{J}^*_{\mathcal{W}_\infty} {=} \mathcal{V}(0) {-} \lim_{t \to \infty}\mathcal{V}(t)$ is the optimal cost of the optimal control law \eqref{OptimalControlProblem}. 
    Finally, adding the term $\gamma^2 \|\zeta(t)\|^2_{\mathcal{L}_2}$ to the right-hand side of \eqref{IneqStability} yields
    \small\begin{align}
    ||z(t)||^2_{\mathcal{W}_{3,2,\bm{\Gamma}}} <& \gamma^2 \|\partial^1\zeta(t)\|^2_{\mathcal{L}_2} + \mathcal{J}^*_{\mathcal{W}_\infty}+ \gamma^2 \|\zeta(t)\|^2_{\mathcal{L}_2}, \\[0.5mm]
         <& \gamma^2 \|\zeta(t)\|^2_{\mathcal{W}_{1,2}} + \mathcal{J}^*_{\mathcal{W}_\infty}.
        \label{IneqStability2}
    \end{align}\normalsize
    Since the optimal control problem is feasible, we have $\mathcal{J}^*_{\mathcal{W}_\infty} \in \mathbb{R}$. Moreover, under the assumption $\zeta(t) \in \mathcal{W}_{1,2}[\mathbb{R}_{\geq 0})$, it follows from \eqref{IneqStability} that there exists $c \in \mathbb{R}_{>0}$ such that $\|z(t)\|^2_{\mathcal{W}_{3,2,\bm{\Gamma}}} < c$. This implies $\lim_{t\to \infty}z(t) = 0$, ensuring the asymptotic stability of the closed-loop system. \qed
\end{proof}



\vspace{-2mm}
\section{NUMERICAL RESULTS}
\label{Sec:Resultados}
\vspace{-2mm}

This section \GR{details} the results of a numerical experiment \GR{conduct} to corroborate the theoretical formulation. \GR{We considered} the dynamic system \eqref{Eq:SubSecModel:HystereticDynamicSystem} \GR{as a model of} a piezoelectric actuator with constants $a_0 = 4$, $a_1 = 2$, and $b_0 = 1$ \citep{Abreu_etal2024_robust}. The hysteresis behavior was modeled by the Prandtl-Ishlinskii operator \eqref{Eq:SubsecModel:PIOperatorVector}\GR{, which incorporates} five play operators\GR{. The} radii \GR{were set as} $r_1 = 0$, $r_2 = 0.63$, $r_3 = 1.27$, $r_4 = 2.54$, and $r_5 = 4.45$, \GR{with weights} $\bm{\theta}=[5.88, 1.58, 0.47, 0.98, 0.40]^T$ \citep{Edardar_etal2014}. \GR{Based on these parameters, we obtained} $\bar{\theta}_{\min} = 5.88$ and $\bar{\theta}_{\max} = 9.31$. The design of the $\mathcal{W}_\infty$ controller was carried out using Theorem \ref{Teorema1}, and the solution of the optimal control problem \eqref{LMIWinf} was obtained with Matlab$^{\copyright}$ using the Yalmip\footnote{\url{https://yalmip.github.io}} 
and Mosek\footnote{\url{https://www.mosek.com}} toolboxes. The controller was initially tuned following Bryson’s method, and then fine-tuned, resulting in $\Gamma_0 {=} 1000$, $\Gamma_1 {=} 45$, $\Gamma_2 {=} 50$, $\Gamma_3 {=} 20$, $\bm{K} = \bm{Y}\bm{X}^{-1} = \begin{bmatrix} 20.24 & 37.58 & 2.26 \end{bmatrix}$, and $\gamma^* = 78.53$.
The control law \eqref{IntOptCtrlLaw} was then applied to system \eqref{Eq:SubSecModel:HystereticDynamicSystem}, and the results are shown in Fig. \ref{Fig:Controle}.

\begin{figure}[htb]
\centering{
\def\svgwidth{1\columnwidth}
\tiny{
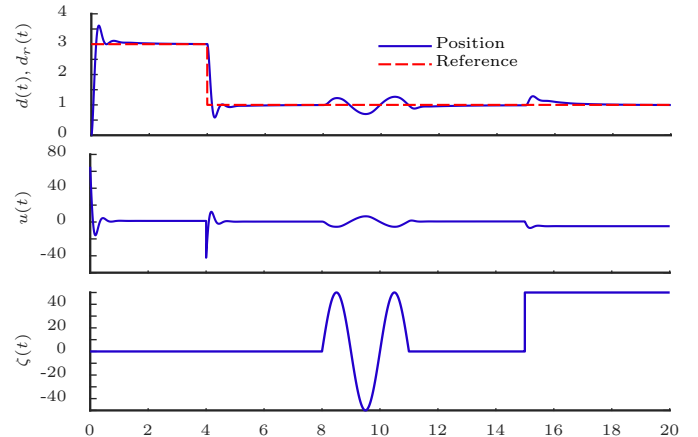\vspace{-2mm}}
\caption{Time evolution of the reference $d_r(t)$, the position $d(t)$ of the piezoelectric actuator\GR{, the control input $u(t)$, and the disturbance $\zeta(t)$.}}
\label{Fig:Controle}
}
\end{figure}

As \GR{demonstrated} in Theorem \ref{Teorema1}, the proposed robust \GR{optimal} $\mathcal{W}_\infty$ controller \GR{successfully attenuated} the effects of hysteresis nonlinearity while ensuring \GR{the} asymptotic stability of the closed-loop system at the desired reference value. Moreover, it \GR{effectively} attenuated sinusoidal disturbances and rejected step-type perturbations in the actuator positioning.

\vspace{-2mm}
\section{Conclusion}
\vspace{-2mm}
\label{Conclusao}

This paper proposed a \GR{robust optimal} $\mathcal{W}_\infty$ controller for \GR{dynamical systems with hysteresis behavior modeled by the} Prandtl--Ishlinskii operator. \GR{By leveraging} weighted Sobolev norms and weak derivatives, \GR{we recast} the \GR{non-differentiable}, input \GR{non-affine} hysteresis as a bounded uncertainty multiplying the input rate, \GR{which enabled a direct} LMI-based synthesis. \GR{We provided a} rigorous stability analysis \GR{for} the resulting closed-loop system, demonstrating that the \GR{proposed} controller guarantees $\mathcal{W}_{3,2,\Gamma}$-stability with an explicit $\mathcal{W}_\infty$-gain bound. 
A numerical study on a piezoelectric actuator validated the theoretical development. The proposed controller \GR{successfully} achieved asymptotic tracking while \GR{simultaneously} attenuating \GR{both} the effects of hysteresis and external disturbances. Future work \GR{will focus on} implementing the proposed control strategy on a physical hysteretic system, extending the formulation to handle systems described by the Bouc–Wen model, and performing \GR{comprehensive} comparative studies against existing \GR{robust control} approaches.


\vspace{-2mm}
\bibliography{references,ifacconf}              
                        
\end{document}